\documentclass[a4paper]{article}

\usepackage{personal_paper}
\usepackage{mathrsfs}
\usepackage{todonotes}

\usepackage[nameinlink,capitalize]{cleveref}

\usepackage{dutchcal}

\newcommand{\tr}{\operatorname{tr}}

\newcommand{\image}{\operatorname{im}}
\newcommand{\kerne}{\operatorname{ker}}
\newcommand {\smat}      [1] {\left[\begin{smallmatrix}{#1}}
\newcommand {\srix}          {\end{smallmatrix}\right]}

\begin{document}

\title{Exact dimension reduction for rough differential equations}

\author{Martin Redmann\thanks{Martin Luther University Halle-Wittenberg, Institute of Mathematics, Theodor-Lieser-Str. 5, 06120 Halle (Saale), Germany, Email: {\tt martin.redmann@mathematik.uni-halle.de}.} \and 
Sebastian Riedel\thanks{FernUniversität in Hagen, Faculty of Mathematics and Computer Science, Chair of Applied Stochastics, 58084 Hagen, Germany, Email: {\tt sebastian.riedel@fernuni-hagen.de}.}
}

\maketitle

\begin{abstract}
In this paper, practically computable low-order approximations of potentially high-dimensional differential equations driven by geometric rough paths are proposed and investigated. In particular, equations are studied that cover the linear setting, but we allow for a certain type of dissipative nonlinearity in the drift as well. In a first step, a linear subspace is found that contains the solution space of the underlying rough differential equation (RDE). This subspace is associated to covariances of linear Ito-stochastic differential equations which is shown exploiting a Gronwall lemma for matrix differential equations. Orthogonal projections onto the identified subspace lead to a first exact reduced order system. Secondly, a linear map of the RDE solution (quantity of interest) is analyzed in terms of redundant information meaning that state variables are found that do not contribute to the quantity of interest. Once more, a link to Ito-stochastic differential equations is used. Removing such unnecessary information from the RDE provides a further dimension reduction without causing an error. Finally, we discretize a linear parabolic rough partial differential equation in space. The resulting large-order RDE is subsequently tackled with the exact reduction techniques studied in this paper. We illustrate the enormous complexity reduction potential in the corresponding numerical experiments.
\end{abstract}



\section*{Introduction}

Rough paths theory is a powerful tool in stochastic analysis that allows to study stochastic ordinary differential equations pathwise. Invented by T.~Lyons in the 90s \cite{Lyo98}, the theory found applications in a variety of fields, cf. \cite{friz_hairer} for an overview. As already conjectured in Lyons' seminal work \cite{Lyo98}, the theory has a vast potential to study stochastic \emph{partial} differential equations (SPDEs), too. Nowadays, there exist numerous approaches to these \emph{rough partial differential equations} (RPDEs). Parabolic equations with roughness in time were studied, e.g., via semigroup theory \cite{GLT06, GT10}, with (stochastic) viscosity theory \cite{CF09, CFO11, FGLS17}, and with a Feynman-Kac approach \cite{DOR15}. Note that this is by far not an exhaustive review of the existing literature, the interesting reader may consult \cite[Chapter 12]{friz_hairer} for a more extensive overview of approaches to rough-in-time RPDEs. Roughness in space of parabolic SPDEs, e.g., in the presence of space-time white noise, was also investigated with rough paths theory \cite{Hai11, HW13}. This line of thinking culminated in Hairer's solution to the KPZ-equation \cite{Hai13} and his seminal theory of regularity structures \cite{Hai14}. We are not trying to summarize the vast literature built on regularity structures here and refer, once again, to \cite{friz_hairer} for a (non-exhaustive) overview. However, when it comes to actually solve rough SPDEs numerically, much less work can be found (let us, however, mention \cite{BBRRS20, Dey12, HM18} here). \smallskip

A standard approach to solve a deterministic (time and space dependent) PDE is to discretize in space and hence to approximate the solution by a high-dimensional system of ordinary differential equations (ODEs). For a RPDE, this strategy results in a system of rough ODEs. Solving these equations numerically is a notoriously difficult problem due to the high dimension of the system, especially if many system evaluations are required. Such computationally challenging situations occur for instance in an optimal control context or if a Monte-Carlo method is used. One common approach in PDE and SPDE theory to escape the curse of dimensionality is to use \emph{model order reduction} (MOR). We refer to \cite{antoulas, morbook2} for a comprehensive overview on various projection-based MOR techniques for deterministic equations and to \cite{redmannbenner, redmannspa2} for a system-theoretic ansatz to tackle high-dimensional stochastic ODEs. The basic observation is that many equations contain redundancies that lead to the fact that the solution described by the system essentially evolves in a subspace (or manifold) of much lower dimension. MOR aims to identify these subspaces (or manifolds) 
on which the dynamics of the equations are essentially acting. Subsequently, one transforms the initial high-dimensional (stochastic) ODE to a (stochastic) ODE of lower order that describes the evolution in this smaller space (or manifold). For many equations, MOR can lead to a drastic dimension reduction while keeping a high accuracy. In fact, MOR is nowadays a standard procedure and widely used in practice. \smallskip

The contribution of this work is to make an important first step towards establishing MOR in the context of rough differential equations (RDEs). More precisely, we will study the exact dimension reduction for a linear RDE driven by a geometric rough path $\mathbf W$, i.e., an equation of the form
\begin{align*}
     dx(t) = Ax(t)\, dt + N\left(x(t)\right)K^{\frac{1}{2}} \, d\mathbf W(t),\quad x(0)=x_0 \in \R^n
\end{align*}
with state space dimension $n$ being large. In fact, we can even allow for a nonlinear drift term, cf. Section \ref{sec:setting}. Our first main result is Theorem \ref{thm_sol_space} that identifies an operator $P$ on $\mathbb R^n$ having the property that every $x(t)$ lies in the image of $P$. Interestingly, $P$ is explicit and given by
\begin{align*}
    P = \int_0^{\infty} \mathbb E[x_B(t) x_B(t)^{\top}]\, dt
\end{align*}
where $x_B$ solves the corresponding Ito stochastic differential equation
\begin{align*}
    dx_B(t) = Ax_B(t)\, dt + N\left(x_B(t)\right)K^{\frac{1}{2}} \, d B(t),\quad x_B(0) = x_0 \in \R^n.
\end{align*}
To prove this theorem, we first approximate $\mathbf W$ by smooth rough paths $\mathbf W^{\epsilon}$ and study the corresponding smooth equations. One key ingredient to make the comparison is a Gronwall-type lemma for matrix differentials, cf. Lemma \ref{lem3}. Once the statement of Theorem \ref{thm_sol_space} is proved for the smooth rough paths $\mathbf W^{\epsilon}$, one can safely pass to the limit using the continuity property of RDEs. The eigenvalue decomposition of $P$ now leads to a dimension reduced equation by using a standard procedure, cf. the discussion after Theorem \ref{thm_sol_space}. If the quantity of interest is given by $y(t) = Cx(t)$ for a matrix $C$, we can potentially reduce the dimension even further, cf. Theorem \ref{thm:quantity_interest_reduced}. In Section \ref{sec:numerical_experiments}, we apply both theorems and perform MOR for a discretized linear RPDE. For the rough heat equation and the quantity of interest being the average temperature on the domain, we can reduce the dimension of the discretized equation from $n = 100$ to $r = 33$ with practically no reduction error. In fact, even a reduction to $r = 5$ yields an error below one percent. This underlines the enormous potential of MOR for RPDEs.

\subsection*{Notation and basic definitions}

Continuous functions $W \colon [0,T] \to \R^d$ will be called \emph{paths}. Let $\alpha \in (0,1]$. If the $\alpha$-H\"older seminorm
\begin{align*}
    \sup_{s < t} \frac{\|W(t) - W(s)\|_2}{|t-s|^{\alpha}}
\end{align*}
is finite, we say $W \in C^{\alpha}$. Here and throughout the rest of the paper, $\|\cdot\|_2$ denotes the Euclidean norm. In the following, we recall some basic definitions of rough paths theory. For a more comprehensive overview, we refer the reader to \cite{friz_hairer, FV10, LCL07}. If $W \colon [0,T] \to \R^d$ is sufficiently smooth, we can define the $n$-times iterated integrals
\begin{align*}
    \mathbb W^{(n)}_{s,t} \coloneqq \int_{s \leq t_1 < \ldots < t_n \leq t} dW(t_1) \otimes \cdots \otimes dW(t_n) \in (\R^d)^{\otimes n} = \bigotimes_{k = 1}^n \R^d.  
\end{align*}
Note that $\mathbb W^{(1)}_{s,t} = W(t) - W(s)$. For some fixed $N$, we call $\mathbf{W} = (\mathbf{W}_{s,t})_{0 \leq s < t \leq T}$ given by
\begin{align*}
    \mathbf{W}_{s,t} = (1,\mathbb W^{(1)}_{s,t}, \ldots, \mathbb W^{(N)}_{s,t}) \in \bigoplus_{n = 0}^N (\R^d)^{\otimes n}
\end{align*}
with $(\R^d)^{\otimes 0} = \R$ the \emph{canonical lift} of $W$. The space $\bigoplus_{n = 0}^N (\R^d)^{\otimes n}$ is called \emph{truncated tensor algebra of level $N$}. Let 
\begin{align*}
    \mathbf{W}_{s,t} = (1,\mathbb W^{(1)}_{s,t}, \ldots, \mathbb W^{(N)}_{s,t}), \quad \tilde{\mathbf{W}}_{s,t} = (1,\tilde{\mathbb{W}}^{(1)}_{s,t}, \ldots, \tilde{\mathbb{W}}^{(N)}_{s,t})
\end{align*}
be two two-parameter functions with values in $\bigoplus_{n = 0}^N (\R^d)^{\otimes n}$. Then, we set
\begin{align*}
    \varrho_{\alpha}(\mathbf W, \tilde{\mathbf{W}}) \coloneqq \sum_{n = 1}^N \sup_{s < t} \frac{\|\mathbb{W}^{(n)}_{s,t} - \tilde{\mathbb{W}}^{(n)}_{s,t} \|}{|t-s|^{n \alpha}}.
\end{align*}
Let $W \in C^{\alpha}$ and $N \leq \frac{1}{\alpha} < N + 1$. A two-parameter function 
\begin{align*}
    \mathbf{W}_{s,t} = (1,\mathbb W^{(1)}_{s,t}, \ldots, \mathbb W^{(N)}_{s,t})
\end{align*}
with $\mathbb W^{(1)}_{s,t} = W(t) - W(s)$ is called a \emph{geometric $\alpha$-H\"older rough path associated to $W$} if there exists a sequence of smooth paths $W^{\epsilon}$ for which the canonical lifts $\mathbf{W}^{\epsilon}$ satisfy
\begin{align*}
    \varrho_{\alpha}(\mathbf W, \mathbf{W}^{\epsilon}) \to 0
\end{align*}
as $\epsilon \to 0$. It can be shown \cite{friz_hairer, FV10} that the set of all geometric rough paths constitutes a complete separable metric space with the metric $\varrho_{\alpha}$. An $\alpha$-H\"older path $x \colon [0,T] \to \R^n$ is called a \emph{solution to the rough differential equation}
\begin{align}\label{eqn:RDE_def}
    dx(t) = b(x(t))\, dt + \sigma(x(t)) \, d \mathbf{W}(t), \quad x(0) = x_0
\end{align}
if $x(0) = x_0$ and for any approximating sequence $\mathbf{W}^{\epsilon}$ to $\mathbf{W}$, the solutions $x^{\epsilon}$ to
\begin{align*}
    dx^{\epsilon}(t) = b(x^{\epsilon}(t))\, dt + \sigma(x^{\epsilon}(t)) \, d {W}^{\epsilon}(t), \quad x^{\epsilon}(0) = x_0
\end{align*}
converge in $\alpha$-H\"older metric to $x$. Conditions on $b$ and $\sigma$ under which \eqref{eqn:RDE_def} has a unique global-in-time solution can be found in \cite[Chapter 10]{FV10}. In particular, it is shown in \cite[Section 10.7]{FV10} that linear equations have unique solutions globally in time.

\section{Setting}\label{sec:setting}
Let $\mathbf W$ be a geometric rough path associated to a path $W\in C^\alpha$ that takes values in $\mathbb \R^d$. By definition, there exists a sequence of smooth paths $W^{\epsilon}$ such that their canonical lifts $\mathbf W^\epsilon$ satisfy $\mathbf W^\epsilon \rightarrow \mathbf W$  ($\epsilon \rightarrow 0$) w.r.t. the rough path metric. In this paper, we will only assume that there exist left-continuous functions $\dot W^\epsilon\in L_T^2$, i.e., $\|\dot W^\epsilon\|_{L^2_T}^2:=\int_0^T \|\dot W^\epsilon(v)\|_2^2 dv<\infty$, so that 
\begin{align}\label{rep_W_eps}
 W^\epsilon(t) =W^\epsilon(0) + \int_0^t \dot W^\epsilon(s) \, ds
\end{align}
for all $\epsilon>0$. We consider the following rough differential equation 
 \begin{subequations}\label{original_system}
\begin{align}\label{stochstatenew}
             dx(t)&=[Ax(t)+ f\left(x(t)\right)] \, dt + N\left(x(t)\right)K^{\frac{1}{2}} \, d\mathbf W(t),\quad x(0)=x_0,\\ \label{output_eq}
            y(t) &= Cx(t),\quad t\in [0, T],
\end{align}
\end{subequations}
with $A\in \mathbb R^{n\times n}$, $C\in \mathbb R^{p\times n}$, $N: \mathbb R^n\rightarrow \mathbb R^{n\times 
d}$ is a linear mapping
defined by $N(x)=\begin{bmatrix}N_1 x &\ldots & N_d x\end{bmatrix}$ for $x\in\mathbb R^n$ and given that $N_1, \ldots, N_d\in\mathbb R^{n\times n}$. Moreover, we interpret the symmetric positive semidefinite matrix $K=(k_{ij})_{i, j=1, \dots, d}$ as a covariance matrix and assume the nonlinearity to be of the form $f(x) = x g(x)$, where $g$ is a scalar function satisfying $g(x)\leq 0$ for all $x\in \mathbb R^n$. This setting covers interesting cases like the cubic function $x\mapsto x-x \|x\|_2^2$ which we can make part of the drift in \eqref{stochstatenew} by setting $g(x)=-\|x\|_2^2$. Note, however, that the classical results on rough differential equations found, e.g., in \cite{FV10} can not be applied here to see that \eqref{stochstatenew} has a unique global-in-time solution since the drift may have superlinear growth. Instead, we can argue as follows: We first consider the corresponding equation without drift, i.e.,
\begin{align}\label{RDE_fully_linear}
     dx(t) &= N\left(x(t)\right)K^{\frac{1}{2}} \, d\mathbf W(t).
\end{align}
The solution to \eqref{stochstatenew} can be obtained by a suitable flow decomposition of \eqref{RDE_fully_linear}, cf. \cite[Section 2]{RS17}. Since \eqref{RDE_fully_linear} is a linear equation, we can use the bounds in \cite[Section 10.7]{FV10} to see that all solution trajectories with initial conditions in a ball with given radius $R > 0$ lie in a compact set $K_R\subset\R^n$. Therefore, for any $R > 0$, we can replace the linear vector fields in \eqref{RDE_fully_linear} by smooth vector fields having compact support by just redefining them to be zero outside $K_R$. Note that $b(x) = Ax + x g(x)$ satisfies \cite[Condition (4.2) and (4.3)]{RS17}. Therefore, we can argue as in \cite[Theorem 4.3]{RS17} to see that the solution to \eqref{stochstatenew} exists globally in time.

We introduce the Lyapunov operator \begin{align}\label{defn_L}
 \mathcal L(X):= A X + X A^\top +\sum_{i, j=1}^d N_i X N_j^\top k_{ij}                                                                          \end{align}
for a simpler notation below, where $X$ is an $n\times n$ matrix.

\section{Approximating solution spaces based on a Gronwall lemma}

Below, we study matrix inequalities that have to be understood in terms of definiteness. In particular, we write $M_1\leq M_2$ for two matrices $M_1$ and $M_2$ if $M_2-M_1$ is a positive semidefinite matrix. Let us first derive such a matrix inequality for a quadratic form of the solution of \eqref{stochstatenew} in case the rough driver is replaced by its smooth approximation.
\begin{lemma}\label{lem1}
Let $x^\epsilon(t)$, $t\in  [0, T]$, satisfy\begin{align}\label{hom_bil_eq}
 dx^\epsilon(t)&=[Ax^\epsilon(t)+ f\left(x^\epsilon(t)\right)]dt+N\left(x^\epsilon(t)\right)K^{\frac{1}{2}}dW^\epsilon(t),\quad x^\epsilon(0)=x_0,
\end{align}
given that $W^\epsilon$ is absolutely continuous with representation in \eqref{rep_W_eps} and left-continuous $\dot W^\epsilon\in L_T^2$. Then, the quadratic form $X^\epsilon(t)=x^\epsilon(t) x^\epsilon(t)^\top \in \mathbb R^{n \times n}$ satisfies
\begin{align}\label{matrix_ineq}
  \dot X^\epsilon(t) \leq \mathcal L(X^\epsilon(t))+ X^\epsilon(t) \left\|\dot W^\epsilon(t)\right\|_2^2, \quad X^\epsilon(0)=x_0 x_0^\top,
\end{align}
for all $t\in  [0, T]$ in which $W^\epsilon$ is differentiable.
\end{lemma}
\begin{proof}
We obtain by the product rule that \begin{align*}
 x^\epsilon(t) x^\epsilon(t)^\top &= x_0 x_0^\top+\int_0^t d x^\epsilon(v) x^\epsilon(v)^\top +   \int_0^t x^\epsilon(v) dx^\epsilon(v)^\top  \\
 &=x_0 x_0^\top+\int_0^t A x^\epsilon(v) x^\epsilon(v)^\top +f\left(x^\epsilon(v)\right)  x^\epsilon(v)^\top +x^\epsilon(v) x^\epsilon(v)^\top A^\top +x^\epsilon(v) f\left(x^\epsilon(v)\right)^\top dv\\
 &\quad + \int_0^t N\left(x^\epsilon(v)\right)K^{\frac{1}{2}}\dot W^\epsilon(v)x^\epsilon(v)^\top + x^\epsilon(v) \dot W^\epsilon(v)^\top K^{\frac{1}{2}} N\left(x^\epsilon(v)\right)^\top dv.
\end{align*}
Now that $t\mapsto x^\epsilon(t) x^\epsilon(t)^\top$ is absolutely continuous, we can take the derivative which exists almost everywhere in points, where $W^\epsilon$ can be differentiated. Subsequently, given two matrices $M_1$ and $M_2$ of suitable dimension, we exploit that $M_1 M_2^\top + M_2 M_1^\top \leq M_1 M_1^\top + M_2 M_2^\top$. In particular, we set  $M_1= N\left(x^\epsilon(v)\right)K^{\frac{1}{2}}$, $M_2= x^\epsilon(v) \dot W^\epsilon(v)^\top$ and use that $f(x) x^\top = x f(x)^\top = x x^\top g(x) \leq 0$. This yields for almost all $t\in [0, T]$ that
\begin{align*}
\frac{d}{dt} x^\epsilon(t) x^\epsilon(t)^\top &\leq A x^\epsilon(t) x^\epsilon(t)^\top  +x^\epsilon(t) x^\epsilon(t)^\top A^\top\\
 &\quad + N\left(x^\epsilon(t)\right)K N\left(x^\epsilon(t)\right)^\top + x^\epsilon(t) \dot W^\epsilon(t)^\top \dot W^\epsilon(t) x^\epsilon(v)^\top.
\end{align*}
The result follows by $N(x) K N(x)^\top= \sum_{i, j=1}^d N_i x x^\top N_j^\top k_{ij}$.
\end{proof}
We now find a (stochastic) representation for the respective equality in \eqref{matrix_ineq} based on a quadratic form of the solution of a linear Ito-stochastic differential equation.
\begin{lemma}\label{lem2}
Let $B$ be a $d$-dimensional standard Brownian motion and $x_B(t)$, $t \geq 0$, be the solution to the following Ito-stochastic differential equation
\begin{align}\label{hom_ito_eq}
 dx_B(t)&=Ax_B(t)dt+N\left(x_B(t)\right)K^{\frac{1}{2}}dB(t),\quad x_B(0)=x_0,
\end{align}
Then, $Z(t)=\mathbb E[x_B(t) x_B(t)^\top]$, $t\geq 0$, solves 
\begin{align}\label{matrix_eq_sinu}
 Z(t) =x_0 x_0^\top+ \int_0^t \mathcal L(Z(v)) dv.
\end{align}
Moreover, given the left-continuous $\dot W^\epsilon\in L_T^2$ from \eqref{rep_W_eps}, the function $ \bar X^\epsilon(t)=\exp\left\{\int_0^t \left\|\dot W^\epsilon(v)\right\|_2^2 dv\right\} Z(t)$, $t\in [0, T]$, solves the following matrix identity:
\begin{align}\label{matrix_eq}
 \bar X^\epsilon(t) =x_0 x_0^\top+ \int_0^t \mathcal L (\bar X^\epsilon(v)) + \bar X^\epsilon(v) \left\|\dot W^\epsilon(v)\right\|_2^2 dv.
\end{align}
\end{lemma}
\begin{proof}
Ito's product rule yields \begin{align*}
x_B(t) x_B(t)^\top = & x_0 x_0^\top+\int_0^t d x_B(v) x_B(v)^\top +   \int_0^t x_B(v) dx_B(v)^\top\\
&+ \int_0^t N\left(x_B(v)\right)K N\left(x_B(v)\right)^\top dv.                           
\end{align*}
We insert \eqref{hom_ito_eq} above, take the expected value and utilize $N(x) K N(x)^\top= \sum_{i, j=1}^d N_i x x^\top N_j^\top k_{ij}$. Since the Ito-integral has mean zero, we obtain \begin{align*}
\mathbb E[x_B(t) x_B(t)^\top] = & x_0 x_0^\top+\int_0^t A \mathbb E[ x_B(v) x_B(v)^\top] dv+   \int_0^t \mathbb E[x_B(v) x_B(v)^\top] A^\top dv\\
&+ \int_0^t \sum_{i, j=1}^d  N_i \mathbb E[x_B(v) x_B(v)^\top] N_j k_{ij} dv                           
\end{align*}
giving us the first part of the claim. Applying the product rule to $\bar X^\epsilon(t)=\exp\left\{\int_0^t \left\|\dot W^\epsilon(v)\right\|_2^2 dv\right\} Z(t)$ and taking \eqref{matrix_eq_sinu} into account, we see that the second part of the result follows.
\end{proof}
\begin{remark}\label{remark_stability}
First, we observe that mean square asymptotic stability, i.e., $\mathbb E\left\| x_B(t)\right\|_2^2\rightarrow 0$ for all $x_0\in\mathbb R^n$ as $t\rightarrow \infty$ is equivalent to  $\mathbb E[ x_B(t)x_B(t)^\top]\rightarrow 0$. For that reason, Lemma \ref{lem2} tells us that mean square asymptotic stability is equivalent to the asymptotic stability of \eqref{matrix_eq_sinu}. It is well known that this is equivalent to 
\begin{align*}
    \lambda(\mathcal L)\subset \mathbb C_- = \{z \in \mathbb C \,:\, \operatorname{Re}(z) < 0 \},
\end{align*} 
where $\lambda(\cdot)$ denotes the spectrum of an operator, and that the decay of the solution of \eqref{matrix_eq_sinu} to zero is exponential. We refer to \cite{damm, staboriginal, redmannspa2} for additional algebraic characterizations and for a further discussion on second moment exponential stability of \eqref{hom_ito_eq}. Let us further point out that this stability concept is stronger than almost sure exponential stability in the linear case, see \cite[Theorem 4.2]{mao}.
\end{remark}
In the next step, a relation between solutions of \eqref{matrix_ineq} and \eqref{matrix_eq} is pointed out. The following lemma can be interpreted as Gronwall type result for matrix differential inequalities/equations. We generalize arguments exploited in \cite{h2_bilinear} in the corresponding proof.
\begin{lemma}\label{lem3}
Suppose that $\dot W^\epsilon\in L_T^2$ is left-continuous. Given an (absolutely) continuous $X^\epsilon(t)$, $t\in [0, T]$, satisfying \eqref{matrix_ineq} and $\bar X^\epsilon(t)$, $t\in [0, T]$, being the solution to the matrix integral equation \eqref{matrix_eq}. Then, 
we have that $X^\epsilon(t)\leq \bar X^\epsilon(t)$ for all $t\in [0, T]$.
\end{lemma}

\begin{proof}
We introduce $Y:= \bar X^\epsilon-X^\epsilon$ and the time-dependent Lyapunov operator $\mathcal L_t(Y):= \mathcal L(Y) + Y \left\|\dot W^\epsilon(t)\right\|_2^2$. From the integrated version of \eqref{matrix_ineq} and \eqref{matrix_eq}, we find that  
\begin{align}\label{Y_inequality}
Y(t)-Y(s) \geq \int_s^t \mathcal L_v(Y(v)) dv, \quad s\leq t.
\end{align}
We define $D(t) := Y(t)-\int_0^t \mathcal L_v(Y(v)) dv$ and consider a perturbed integral equation 
\begin{align}\label{perturbed_equation}
Y_\gamma (t) = \gamma I + \int_0^t [ \mathcal L_v(Y_\gamma(v)) +\gamma I] dv + D(t) 
\end{align}
with parameter $\gamma\geq 0$. By construction, we observe that $Y_0(t)= Y(t)$ for all $t\in [0, T]$. Moreover, it holds that $\lim_{\gamma \rightarrow 0} Y_\gamma(t) = Y(t)$ for all $t\in [0, T]$.

Below, let us assume that $Y_\gamma$ is not positive definite for $\gamma>0$ meaning that $\tilde z^\top Y_\gamma(\tilde t) \tilde z \leq 0$ for some $\tilde z\ne 0$ and  $\tilde t>0$.
$Y_\gamma$ is positive definite at $t=0$ as $Y_\gamma(0) = \gamma I$. This is equivalent to all  the eigenvalues of this matrix being positive. Now that $Y_\gamma$ is continuous and takes values in the space of symmetric matrices, there exist continuous and real functions $\lambda_1, \ldots, \lambda_n$ such that $\lambda_1(t), \ldots, \lambda_n(t)$ represent the eigenvalues of $Y_\gamma(t)$ for each fixed $t\in [0, T]$, see \cite[Corollary VI.1.6]{Bhatia1997}. By assumption, at least one of these eigenvalue functions crosses or touches zero, while starting with a positive value. Let $\lambda_i$ be the one that reaches zero first at some $t_0\in (0, \tilde t]$, i.e., $t_0$ is the smallest point of time with $\lambda_i(t_0)=0$.
Since we have $\lambda_i(t_0)=0$ while all the other eigenvalues are nonnegative, $Y_\gamma$ turns from a positive definite into a positive semidefinite matrix at this $t_0$ meaning that
\begin{align}\label{contradicttionprop}
z_0^\top Y_\gamma(t_0) z_0 = 0 \quad \text{and} \quad z_0^\top Y_\gamma(t) z_0 > 0,\quad 0\leq t<t_0,
\end{align}
for some $z_0 \ne 0$ while $z^\top Y_\gamma(t_0) z\geq 0$ for all $z\in \mathbb R^n$.  
Now, $\mathcal L_t$ is a Lyapunov operator for fixed $t\geq 0$ and hence resolvent positive, see Appendix \ref{appendixbla}.
The relation $0 = z_0^\top Y_\gamma(t_0) z_0 = \langle Y_\gamma(t_0), z_0z_0^\top\rangle_F$  consequently implies $0\leq \langle \mathcal L_{t_0}(Y_\gamma(t_0)), z_0z_0^\top\rangle_F = z_0^\top \mathcal L_{t_0}(Y_\gamma(t_0))z_0$ according to Theorem \ref{equiresolpos}.
As $\dot W^\epsilon$ is left-continuous, the same holds for $t\mapsto z_0^\top \mathcal L_{t}(Y_\gamma(t))z_0$. For that reason, there exists a $\delta > 0$ such that $z_0^\top \mathcal L_{v}(Y_\gamma(v))z_0>-\gamma \left\|z_0\right\|_2^2$ for all $v \in (t_0 - \delta,t_0]$. Let $s , t \in (t_0 - \delta,t_0]$ with $s \leq t$. Then,
\begin{align*}
    z_0^\top Y_\gamma(t) z_0 - z_0^\top Y_\gamma(s) z_0 &= \int_s^t z_0^\top \mathcal L_{v}(Y_\gamma(v)) z_0 + \gamma \|z_0 \|_2^2 \, dv + z_0^\top(D(t) - D(s)) z_0 \\
    &\geq z_0^\top(D(t) - D(s)) z_0.
\end{align*}
From \eqref{Y_inequality}, we obtain that $D(t) - D(s) \geq 0$. Consequently, we know that $z_0^\top Y_\gamma(s) z_0 \leq z_0^\top Y_\gamma(t) z_0$, i.e., $v \mapsto z_0^\top Y_\gamma(v) z_0$ is increasing on $(t_0 - \delta,t_0]$ which contradicts \eqref{contradicttionprop}.
Therefore, 
$Y_\gamma(t)$ is positive definite for all $t\in [0, T]$ and $\gamma>0$. Taking the limit of $\gamma \rightarrow 0$, we obtain $Y(t)\geq 0$ for all $t\in [0, T]$ which concludes the proof.
\end{proof}
As a consequence of Gronwall Lemma  \ref{lem3}, the following theorem can be established that provides information on the solution space of the considered rough differential equation.

\begin{theorem}\label{thm_sol_space}
Suppose that $x$ is the solution of \eqref{stochstatenew} on $[0, T]$ with a driver $\mathbf W$. 
Then, it holds that \begin{align}\label{sol_space}
 x(t) \in \image[P_T], \quad t\in [0, T],                                                                                                                                                \end{align}
where $P_T = \int_0^T \mathbb E[x_B(t) x_B(t)^\top] dt$ with $x_B$ solving the Ito-stochastic differential equation \eqref{hom_ito_eq}. If \eqref{hom_ito_eq} further is mean square asymptotically stable, that is, $\mathbb E\left\| x_B(t)\right\|_2^2\rightarrow 0$ as $t\rightarrow \infty$, the limit $P:=\lim_{T\rightarrow \infty} P_T$ exists. Then, $P_T$ can be replaced by $P$ in \eqref{sol_space}.
\end{theorem}
\begin{proof}
Let $z\in \kerne[P_T]$ and let $x^\epsilon$ be the approximation of $x$ defined by \eqref{hom_bil_eq}.  Then, \begin{align*}
\int_0^T \langle z, x^\epsilon(t) \rangle_2^2 dt = z^\top  \int_0^T x^\epsilon(t)x^\epsilon(t)^\top dt\, z.                          
\end{align*}
By Lemma \ref{lem1}, we observed that $x^\epsilon(t)x^\epsilon(t)^\top$ is a continuous solution to \eqref{matrix_ineq}. Hence, it can be bounded from above by the solution $\bar X^\epsilon$ to \eqref{matrix_eq} using Lemma \ref{lem3}. By Lemma \ref{lem2}, it is known that $\bar X^\epsilon(t) =\exp\left\{\int_0^t \left\|\dot W^\epsilon(v)\right\|_2^2 dv\right\} \mathbb E[x_B(t) x_B(t)^\top]$. Consequently, we have \begin{align}\nonumber
\int_0^T \langle z, x^\epsilon(t) \rangle_2^2 dt &\leq \exp\left\{\int_0^T \left\|\dot W^\epsilon(v)\right\|_2^2 dv\right\} z^\top  \int_0^T  \mathbb E[x_B(t) x_B(t)^\top] dt\, z\\ \label{kernel_orth}
&=\exp\left\{\int_0^T \left\|\dot W^\epsilon(v)\right\|_2^2 dv\right\} z^\top P_T z = 0.                          
\end{align}
Since $x^\epsilon$ is continuous it follows that $\langle z, x^\epsilon(t) \rangle_2^2 = 0$, $t\in [0, T]$. Taking the limit as $\epsilon\rightarrow 0$, we find $\langle z, x(t) \rangle_2^2 = 0$ for all $t\in [0, T]$. This means that $x(t)$ is orthogonal to $\kerne[P_T]$. By the symmetry of $P_T$ the orthogonal complement of this kernel is $\image[P_T]$, so that the first claim follows. If the Ito-stochastic differential equation is mean square asymptotically stable, it decays exponentially fast to zero, see Remark \ref{remark_stability}. This implies exponential convergence of $\mathbb E[x_B(t) x_B(t)^\top]$ to zero. In this case, $P$ exists and it holds that $z^\top P_T z \leq z^\top P z$ for all $z\in \mathbb R^n$. Now, choosing $z\in \kerne[P]$, the second claim follows from \eqref{kernel_orth}. This concludes the proof.
\end{proof}

We can now consider the eigenvalue decomposition of $\mathcal P\in \{P_T, P\}$ given by
\begin{align*}
          \mathcal P= \begin{bmatrix}
              V_{\mathcal P} &\star
             \end{bmatrix}
 \begin{bmatrix}{\Lambda}& 0\\ 
0 &0\end{bmatrix} \begin{bmatrix}
              V_{\mathcal P}^\top \\ \star
             \end{bmatrix}= V_{\mathcal P} \Lambda V_{\mathcal P}^\top,                                                                                \end{align*}
where $\Lambda$ is the diagonal matrix of non-zero eigenvalues of $\mathcal P$ and the matrix $V_{\mathcal P}$ of associated eigenvectors provides an orthonormal basis for $\image[\mathcal P]$. Therefore, we can find a reduced order function $x_r(t)\in \mathbb R^r$, with $r$ being the number of non-zero eigenvalues, giving us $x(t) = V_{\mathcal P} x_r(t)$. Inserting this identity into \eqref{original_system} and multiplying the resulting equation with $V_{\mathcal P}^\top$ from the left leads to  \begin{subequations}\label{red_system}
\begin{align}\label{red_state}
             dx_r(t)&=[V_{\mathcal P}^\top A V_{\mathcal P}x_r(t)+ V_{\mathcal P}^\top f\left(V_{\mathcal P} x_r(t)\right)]dt\\ \nonumber
             &\quad+V_{\mathcal P}^\top N\left(V_{\mathcal P} x_r(t)\right)K^{\frac{1}{2}}d\mathbf W(t),\quad x_r(0)=V_{\mathcal P}^\top x_0,\\ \label{red_output_eq}
            y(t) &= CV_{\mathcal P} x_r(t),\quad t\in [0, T],
\end{align}
\end{subequations}      

\section{Redundancies in the quantity of interest} \label{sec_out_red}

Instead of looking at an approximation for the solution space of the state variable, let us now point out which states in $x$ can be removed from the dynamics without an effect on $y$ defined in \eqref{output_eq}. This allows us to reduce the dimension of \eqref{red_system} further. Here, we assume a purely linear system, i.e., $f\equiv 0$ in \eqref{stochstatenew}. Let $\mathcal Z$ denote the solution to \begin{align}\label{matrix_eq_sinu_dual}
 \mathcal Z(t) =C^\top C+ \int_0^t \mathcal L^*(\mathcal Z(v)) dv.
\end{align}
which can be interpreted as the dual equation of \eqref{matrix_eq_sinu}, where \begin{align}\label{defn_L_ad}
 \mathcal L^*(X):= A^\top X + X A +\sum_{i, j=1}^d N_i^\top X N_jk_{ij}.                                       \end{align}
Here, $\mathcal L^*$ is the adjoint operator of $\mathcal L$ with respect to the Frobenius inner product $\langle \cdot, \cdot\rangle_F$. As $\mathcal Z=\mathcal Z(\cdot, C^\top C)$ is linear in its initial state, we obtain $\mathcal Z(t, C^\top C)= \sum_{\ell=1}^p=\mathcal Z(t, c_\ell^\top c_\ell)$, where $c_\ell$ is the $\ell$th row of $C$. By Lemma \ref{lem2}, we know that $\mathcal Z(t, c_\ell^\top c_\ell)=\mathbb E[{\mathbcal x}_{B}(t){\mathbcal x}_{B}(t)^\top]$, where ${\mathbcal x}_{B}$ solves the Ito-stochastic differential equation
\begin{align}\label{dual_sde}
 d{\mathbcal x}_{B}(t)&=A^\top {\mathbcal x}_{B}(t)dt+N^*\left({\mathbcal x}_{B}(t)\right)K^{\frac{1}{2}}dB(t),\quad {\mathbcal x}_{B}(0)=c_\ell^\top,
\end{align}
with $N^*(x):=\begin{bmatrix}N_1^\top x &\ldots & N_d^\top x\end{bmatrix}$. This stochastic representation implies that $\mathcal Z(t)$ is a positive semidefinite matrix for all fixed $t$. This is exploited in the next lemma.
\begin{lemma}\label{lemma_kernel}
Let us define $Q_T:=\int_0^T \mathcal Z(t)dt$, where $\mathcal Z$ solves \eqref{matrix_eq_sinu_dual}. Suppose that $z\in \kerne[Q_T]$. Then, we have \begin{align}\label{kernel_preservation}
 Q_T Az = 0, \quad Cz= 0,\quad (K\otimes Q_T) \smat N_1^\top&\dots &N_d^\top\srix^\top z  =0.                                            \end{align}
If the solution of \eqref{dual_sde} satisfies $\mathbb E\left\| {\mathbcal x}_B(t)\right\|_2^2\rightarrow 0$ as $t\rightarrow \infty$, the limit $Q:=\lim_{T\rightarrow \infty} Q_T$ exists and \eqref{kernel_preservation} holds when replacing $Q_T$ by its limit.
\end{lemma}
\begin{proof}
Using \eqref{matrix_eq_sinu_dual} for $t=T$ and the linearity of $\mathcal L^*$, we obtain \begin{align}\nonumber
 \mathcal Z(T) &=C^\top C+ \mathcal L^*(Q_T)=C^\top C+  A^\top Q_T + Q_T A +\sum_{i, j=1}^d N_i^\top Q_T N_jk_{ij}\\ \label{eq_QT}
 &=C^\top C+ A^\top Q_T + Q_T A + \smat N_1^\top&\dots &N_d^\top\srix(K\otimes Q_T) \smat N_1^\top&\dots &N_d^\top\srix^\top,
\end{align}
since $\sum_{i, j=1}^d N_i^\top k_{ij} Q_T N_j  = \smat N_1^\top&\dots &N_d^\top\srix(K\otimes Q_T) \smat N_1^\top&\dots &N_d^\top\srix^\top$. Suppose that $z\in \kerne[Q_T]$. Then, we have $0=z^\top Q_T z = \int_0^T \left\|\mathcal Z(t)^{\frac{1}{2}} z\right\|_2^2dt$ exploiting that $\mathcal Z(t)$ is positive semidefinite. As $\mathcal Z$ is continuous, we obtain that $\mathcal Z(t)z=0$ for all $t\in [0, T]$. Now, we can multiply \eqref{eq_QT} with $z^\top$ from the left and $z$ from the right yielding \begin{align}\label{kernel_id}
0=z^\top C^\top C z+ z^\top \smat N_1^\top&\dots &N_d^\top\srix(K\otimes Q_T) \smat N_1^\top&\dots &N_d^\top\srix^\top z.
\end{align}
$K\otimes Q_T$ is a positive semidefinite matrix, because $K$ and $Q_T$ are positive semidefinite. Hence, both summands on the right-hand side of \eqref{kernel_id} must be zero. Therefore, we have $Cz=0$ and $(K\otimes Q_T) \smat N_1^\top&\dots &N_d^\top\srix^\top z =0$. With this knowledge, we multiply \eqref{eq_QT} only with $z$ from the right resulting in $Q_T A z =0$. Finally, $\mathbb E\left\| {\mathbcal x}_B(t)\right\|_2^2\rightarrow 0$ is equivalent to $\mathbb E[{\mathbcal x}_B(t){\mathbcal x}_B(t)^\top]\rightarrow 0$ as $t\rightarrow \infty$. In particular, this convergence is exponential, see Remark \ref{remark_stability}. Therefore, $\mathcal Z$ converges exponentially fast to zero yielding the existence of $Q=\int_0^\infty \mathcal Z(t) dt$. Taking the limit of $T\rightarrow \infty$ in \eqref{eq_QT}, this $Q$ satisfies $0=C^\top C + \mathcal L^*(Q)$, so that the above arguments can be used to proof the same result for $Q$ instead of $Q_T$.
\end{proof}
Notice that mean square asymptotic stability of \eqref{dual_sde} exploited in Lemma \ref{lemma_kernel} is equivalent to the same type of stability in \eqref{hom_ito_eq} since $\lambda(\mathcal L^*)=\lambda(\mathcal L)$, see Remark \ref{remark_stability}. Let us introduce $\mathcal Q\in \{Q_T, Q\}$. Since $\mathcal Q$ is positive semidefinite, we can find an associated orthogonal basis for $\mathbb R^n$ consisting of eigenvectors $(q_k)$ of $\mathcal Q$. We define the matrix $V_{\mathcal Q}:=\smat q_1&\dots &q_r\srix$, where the columns of this matrix are the eigenvectors corresponding to the non zero eigenvalue of $\mathcal Q$. The remaining eigenvectors $q_{r+1}, \dots, q_n$ form a basis for $\kerne[\mathcal Q]$. We set $V_{\mathcal Q}^\perp:=\smat q_{r+1}&\dots &q_n\srix$. We can find processes ${\mathbcal x}_r$ and $\tilde{{\mathbcal x}}$, so that $x(t)=  V_{\mathcal Q}{\mathbcal x}_r(t) + V_{\mathcal Q}^\perp \tilde{{\mathbcal x}}(t)$ which implies that ${\mathbcal x}_r(t) =  V_{\mathcal Q}^\top x(t)$. As a consequence of Lemma \ref{lemma_kernel}, we obtain that $y(t) = Cx(t) = C V_{\mathcal Q}{\mathbcal x}_r(t)$. Now, the differential equation associated to ${\mathbcal x}_r$ is obtained by \begin{align*}
 d{\mathbcal x}_r(t) =     V_{\mathcal Q}^\top d x(t) = V_{\mathcal Q}^\top Ax(t)dt+V_{\mathcal Q}^\top N\left(x(t)\right)K^{\frac{1}{2}}d\mathbf W(t),\quad {\mathbcal x}_r(0)=V_{\mathcal Q}^\top x_0.                                \end{align*}
By Lemma \ref{lemma_kernel}, we have $V_{\mathcal Q}^\top Ax(t) = V_{\mathcal Q}^\top A(V_{\mathcal Q}{\mathbcal x}_r(t) + V_{\mathcal Q}^\perp \tilde{{\mathbcal x}}(t))= V_{\mathcal Q}^\top A V_{\mathcal Q}{\mathbcal x}_r(t)$. Moreover, given that the covariance matrix $K$ is invertible, we can multiply the last identity of \eqref{kernel_preservation} with $K^{-1} \otimes I$ providing $\mathcal Q N_i z=0$ for $z\in \kerne[\mathcal Q]$ and all $i=1, \dots, d$. This can now be exploited to obtain that $V_{\mathcal Q}^\top N\left(x(t)\right)= V_{\mathcal Q}^\top N\left(V_{\mathcal Q}{\mathbcal x}_r(t)\right)$. Let us summarize the above considerations in the following theorem.
\begin{theorem}\label{thm:quantity_interest_reduced}
Given $\mathcal Q\in \{Q_T, Q\}$ defined in Lemma \ref{lemma_kernel},  $K$ being invertible and $f\equiv 0$. 
Then, we find a reduced order system with the same quantity of interest like \eqref{original_system}. It is given by \begin{subequations}\label{red_system_Q}
\begin{align}\label{red_state_Q}
             d{\mathbcal x}_r(t)&=V_{\mathcal Q}^\top A V_{\mathcal Q}{\mathbcal x}_r(t) dt+V_{\mathcal Q}^\top N\left(V_{\mathcal Q} {\mathbcal x}_r(t)\right)K^{\frac{1}{2}}d\mathbf W(t),\quad {\mathbcal x}_r(0)=V_{\mathcal Q}^\top x_0,\\ \label{red_output_eq_Q}
            y(t) &= CV_{\mathcal Q} {\mathbcal x}_r(t),\quad t\in [0, T],
\end{align}
\end{subequations} 
with $V_{\mathcal Q}:=\smat q_1&\dots &q_r\srix$, where $q_1, \dots, q_r$ are orthonormal eigenvectors of $\mathcal Q$ corresponding to all $r$ non zero eigenvalues.
\end{theorem}

\section{Numerical experiments}\label{sec:numerical_experiments}

Let the regularity of $\mathbf W=(W, \mathbb W)$ now be $\alpha\in(1/3, 1/2]$. As before, we assume that it can be approximate (w.r.t. the rough path metric) by the lift of $W^\epsilon$ with representation \eqref{rep_W_eps}. 

\subsection{Linear rough PDEs and Feynman-Kac solutions}

 We aim to study the solution 
 \begin{align*}
     [0,T] \times \R^m \ni (t,x) \mapsto u(t,x)
 \end{align*} 
 to the initial value problem
\begin{align}\label{RPDE}
       du = L(u) \, dt + \sum_{k=1}^d \Gamma_k(u) \, d\mathbf{W}_k,\quad u(0, \cdot) = g,
  \end{align}
where $L$ and $\Gamma = (\Gamma_1, \ldots, \Gamma_d)$ are 
\begin{align*}
  Lh(\zeta) &:= \frac{1}{2} \tr\left( \sigma(\zeta) \sigma(\zeta)^\top D^2 h(\zeta) \right) +
  \left\langle b(\zeta), Dh(\zeta) \right\rangle+c(\zeta) h(\zeta),\\ 
    \Gamma_k h(\zeta) &:= \left\langle \beta_k(\zeta), Dh(\zeta)\right\rangle+ \gamma_k(\zeta)h(\zeta),
\end{align*}
for a suitable test function $h:\mathbb R^m \rightarrow \mathbb R$.  For the Feynman-Kac approach, it will be convenient to apply the time change $t \mapsto T - t$ and to study the equivalent terminal value problem
\begin{align}\label{RPDE_terminal}
       - dv = L(v) \, dt + \sum_{k=1}^d \Gamma_k(v) \, d \overleftarrow{\mathbf{W}}_k,\quad v(T, \cdot) = g,
  \end{align}
instead where $ \overleftarrow{\mathbf{W}}$ denotes the time reversed rough path $\overleftarrow{\mathbf{W}}(t) = \mathbf{W}(T-t)$. If we replace $\overleftarrow{\mathbf{W}}$ by $\overleftarrow{W}^\epsilon$, then every bounded $v^\epsilon\in C^{1,2}([0, T]\times\mathbb R^m, \mathbb R)$ solution to the PDE (driven by $\overleftarrow{W}^\epsilon$) has the Feynman-Kac representation
 \begin{align} \label{feynman_kac_rep}
    v^\epsilon(t,\zeta) = \mathbb E\left[ g(x^\zeta(T)) \exp\left( \int_t^T c(x^\zeta(s))\, ds + \int_t^T \gamma(x^\zeta(s)) \dot{\overleftarrow{W}^{\epsilon}}(s) \, ds  \right) \right], 
        \end{align}       
  where $x^\zeta$ is the solution to the Ito-stochastic differential equation
 \begin{align*}
    dx^\zeta(s) &= b\left((x^\zeta(s)\right) \, ds + \begin{bmatrix}\sigma\left(x^\zeta(s)\right)  \beta\left(x^\zeta(s)\right)\end{bmatrix} 
\begin{bmatrix}
  dB(s) \\ d\overleftarrow{W}^\epsilon(s)
 \end{bmatrix}, \quad t \leq s \leq T, \\
 x^\zeta(t) &= \zeta,
\end{align*}
If $v^\epsilon\not\in C^{1,2}([0, T]\times\mathbb R^m, \mathbb R)$, we use \eqref{feynman_kac_rep} to define the solution of the PDE as long as the associated stochastic differential equation admits a unique solution. Now, given that the initial value $g$ is continuous and bounded, \cite{phd_diehl, DOR15, friz_hairer} showed that for $v^\epsilon$ in \eqref{feynman_kac_rep}, it holds that 
  \begin{align}\label{FK_sol_RPDE}
 v^\epsilon(t,\zeta)\rightarrow   v(t,\zeta) \coloneqq \mathbb E\left[ g(x^\zeta(T)) \exp\left( \int_t^T c(x^\zeta(s)) \, ds + \int_t^T \gamma(x^\zeta(s))  \, d\overleftarrow{\mathbf W}(s)  \right) \right], 
        \end{align}       
point-wise in time and space. Here, $x^\zeta$ is the solution to the rough differential equation
 \begin{align}
    \begin{split}
        dx^\zeta(s) &= b\left((x^\zeta(s)\right) \, ds + \begin{bmatrix}\sigma\left(x^\zeta(s)\right) & \beta\left(x^\zeta(s)\right)\end{bmatrix} 
d\mathbf{Z}(s), \quad  t \leq s \leq T, \\
x^\zeta(t) &= \zeta,
    \end{split}
    \label{eqn:hybrid_rde}
\end{align}
with the joint rough path $\mathbf{Z} = \left(Z, \mathbb{Z}\right)$ \begin{equation*}
  Z(t) :=
  \begin{pmatrix}
    B(t)\\
    \overleftarrow{W}(t)
  \end{pmatrix},\quad
  \mathbb{Z}_{s,t} :=
  \begin{pmatrix}
    \int_s^t (B(v)-B(s))\otimes dB(v) & \int_s^t (\overleftarrow{W}(v) - \overleftarrow{W}(s)) \otimes dB(v)\\
    \int_s^t (B(v)-B(s))\otimes d\overleftarrow{W}(v) & \overleftarrow{\mathbb{W}}_{s,t}
  \end{pmatrix},
  \end{equation*}
  where the stochastic integrals are understood as Ito-integrals. The limit $v$ now defines the solution to \eqref{RPDE} given that \eqref{eqn:hybrid_rde} has a unique solution. 

\subsection{Dimension reduction for spatially discretized rough heat equations}

We specify the coefficients for our numerical experiments by setting $\sigma\equiv \sqrt{2} I$, $b\equiv 0$ and $c\equiv 0$ resulting in the rough heat equation \begin{align}\label{heat_RPDE}
       du(t, \zeta) = \Delta u(t, \zeta) \, dt + \sum_{k=1}^d \big(\left\langle \beta_k(\zeta), \nabla u(t, \zeta)\right\rangle +\gamma_k(\zeta) u(t, \zeta)\big) d\mathbf{W}_k(t),\quad u(0, \cdot) = g.
  \end{align}
Instead of exploiting the Feynman-Kac representation in \eqref{FK_sol_RPDE}, we formally discretize \eqref{RPDE} by a finite difference scheme. Moreover, we consider the bounded spatial domain $[0, 1]^m$ (in contrast to the above Feynman-Kac theory). Here, we set additional boundary conditions which are assumed to be of Dirichlet type. Notice that  equation \eqref{heat_RPDE} can then also be defined in the mild sense (for general non geometric drivers) when the transport term is absent, see \cite{friz_hairer, GT10}.\smallskip

For simplicity let us set $m=1$. Then, $h_\zeta:=\frac{1}{(n+1)}$ is supposed to be the spatial step size parameter leading to a grid $\zeta_j = j h_\zeta$ for $j=0, 1, \dots, n+1$. Intuitively, we find that  $x_j(t)\approx u(t, \zeta_j)$, where 
 \begin{equation}    \label{finite_difference_model}                                                       
\begin{aligned}
dx_1(t) &= \frac{x_2(t)-2x_1(t)}{h_\zeta^2} dt + \sum_{k=1}^d \Big(\beta_k(\zeta_1)\frac{x_2(t)-x_1(t)}{h_\zeta}+\gamma_k(\zeta_1) x_1(t)\Big) d\mathbf W_k(t),\\
dx_j(t) &= \frac{x_{j+1}(t)-2x_j(t)+x_{j-1}(t)}{h_\zeta^2} dt + \sum_{k=1}^d \Big(\beta_k(\zeta_j)\frac{x_{j+1}(t)-x_j(t)}{h_\zeta}+ \gamma_k(\zeta_j) x_j(t)\Big) d\mathbf W_k(t), \\
dx_n(t) &= \frac{-2x_n(t)+x_{n-1}(t)}{h_\zeta^2} +\sum_{k=1}^d \Big(\beta_k(\zeta_n)\frac{-x_n(t)}{h_\zeta}+ \gamma_k(\zeta_n) x_n(t)\Big) d\mathbf W_k(t)                  \end{aligned}    
\end{equation}
for $j\in\{2, \dots, n-1\}$ taking into account that $u(t, 0) = u(t, 1)=0$. The initial condition associated to \eqref{finite_difference_model} is $x(0)=\left(\begin{matrix}                                              g(\zeta_1)&\dots & g(\zeta_n)         \end{matrix}\right)^\top$. 
       $W$ shall now be $2$-dimensional, where its components are paths of independent fractional Brownian motions with Hurst index $H=0.4$. Further, let us set $n=100$, $\gamma_1(\zeta)= 4 \sin(\zeta)$, $\gamma_2(\zeta)= 4 \cos(\zeta)$, $\beta_1\equiv 0.4$, $\beta_2\equiv -0.2$ and $g(\zeta)=\expn^{-2\vert \zeta-0.5\vert^2}$, $\zeta\in [0, 1]$. We fix $T=0.5$ and introduce the quantity of interest \begin{align}\label{rec_dif_output}
  y(t) = \frac{1}{n} \sum_{j=1}^n x_j(t)
  \end{align}
being the average temperature, i.e., $C=\frac{1}{n}\begin{bmatrix}1 &\ldots & 1\end{bmatrix}$. We illustrate $y$ in Figure \ref{output1} given the driver depicted in Figure \ref{driverW}.\begin{figure}[ht]
 \begin{minipage}{0.45\linewidth}
  \hspace{-0.5cm}
 \includegraphics[width=1.1\textwidth,height=5.5cm]{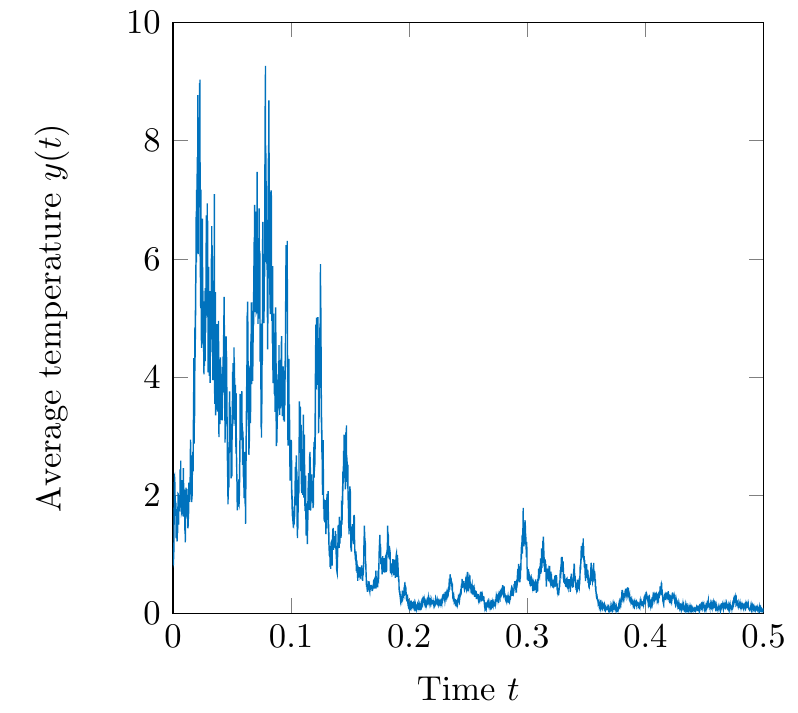}
 \caption{Output $y$ in \eqref{rec_dif_output} of \eqref{finite_difference_model}  with driver in Figure \ref{driverW}.}\label{output1}
 \end{minipage}\hspace{0.5cm}
 \begin{minipage}{0.45\linewidth}
  \hspace{-0.5cm}
 \includegraphics[width=1.1\textwidth,height=5.5cm]{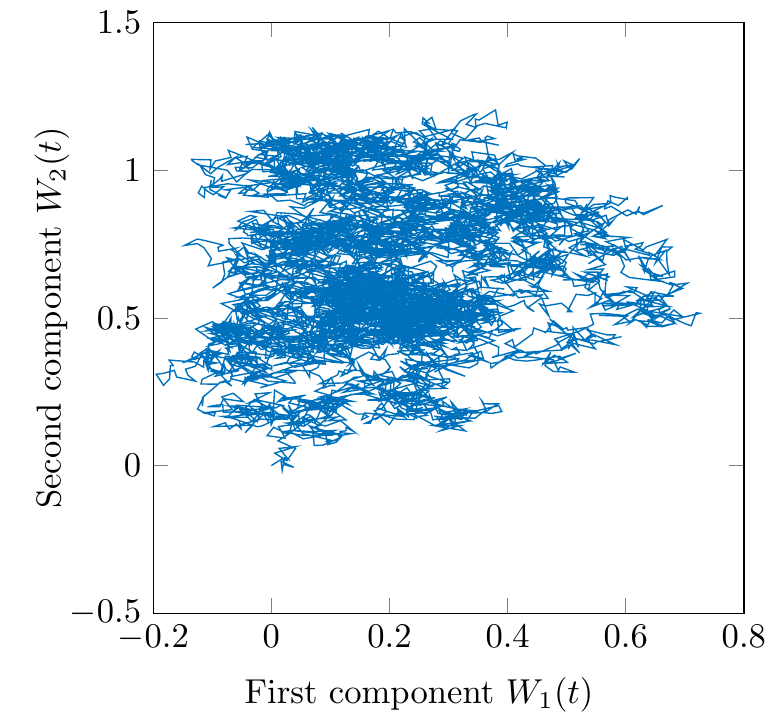}
 \caption{Path of a $2$D fractional Brownian motion used as driver.}\label{driverW}
 \end{minipage}
 \end{figure}  
Consequently, \eqref{finite_difference_model}  together with \eqref{rec_dif_output} yield a system of the form \eqref{original_system} with $f\equiv 0$. Moreover, notice that \eqref{finite_difference_model} is a mean square asymptotically stable system given the above parameters. Therefore, $P$ and $Q$, introduced in Theorem \ref{thm_sol_space} and Lemma \ref{lemma_kernel},  exist and can be used to identify unnecessary information. In particular, $P$ and $Q$ can be computed much easier than $P_T$ and $Q_T$. We obtain them from solving  $0=x_0 x_0^\top+\mathcal L(P)$ and $0=C^\top C+\mathcal L^*(Q)$ which are the equations derived by taking the limit as $t\rightarrow \infty$ in \eqref{matrix_eq_sinu} and \eqref{matrix_eq_sinu_dual}. We observe from Figure \ref{eigPQ} that $P$ and $Q$ have many eigenvalues below machine precision that are numerically zero. 
\begin{figure}
\center
 \includegraphics[width=0.75\textwidth,height=6cm]{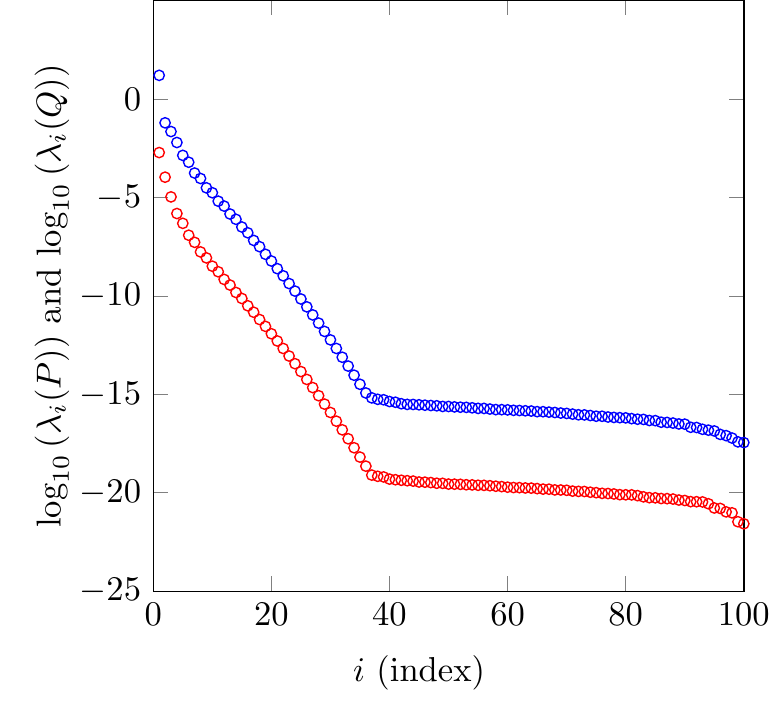}
 \caption{Logarithmic eigenvalues of $P$ (blue) and $Q$ (red).}\label{eigPQ}
 \end{figure} 
As a first step we remove the ones of $P$ resulting in a reduced model \eqref{red_system} of order $35$. Subsequently, the dimension of this system can be lowered further by applying the procedure of Section \ref{sec_out_red}. Here, two eigenvalues below machine precision can be detected finally providing a model of dimension $r=33$ in which we do not expect any reduction error. However, it is important to notice that there are several sources of numerical errors like, for instance, the time discretization leading to a non zero error in practice. For that reason, we denote the output of the reduced system by $y_r$ and find a relative $L^2_T$-error  $\frac{\left\| y-y_r\right\|_{L^2_T}}{\left\|y\right\|_{L^2_T}}= 1.5710$e$-14$ for $r=33$. This can be assumed to be an exact approximation neglecting the other numerical errors. In addition, the logarithm of the point-wise error $\frac{\left\vert  y(t)-y_r(t)\right\vert}{\left\vert y(t)\right\vert}$ for the same setting is shown in Figure \ref{pointwise_error}. Finally, we conducted experiments related to dimension reduction with a true error. In detail, in addition to the (numerical) zero eigenvalues, we neglect eigenspaces of $P$ and $Q$ that are associated to very small eigenvalues of which we have many, see Figure \ref{eigPQ}. This is motivated by an observation in Ito-SDE settings, where those direction have a tiny influence on the dynamics, see, e.g, \cite{redmannspa2}. Figure \ref{L2error} depicts the relative $L^2_T$-errors for  $r=5, 7, 9, \dots, 33$ in logarithmic scale.
We observe a small error in each case, e.g., of order $1$e$-08$ for an $r$ around $20$, Moreover, the deviation from the true output is below one percent even for $r=5$. This illustrates that rough differential equations can have a very high reduction potential beyond truncating state variables that have no contribution.  
\begin{figure}[ht]
 \begin{minipage}{0.45\linewidth}
  \hspace{-0.5cm}
 \includegraphics[width=1.1\textwidth,height=5.5cm]{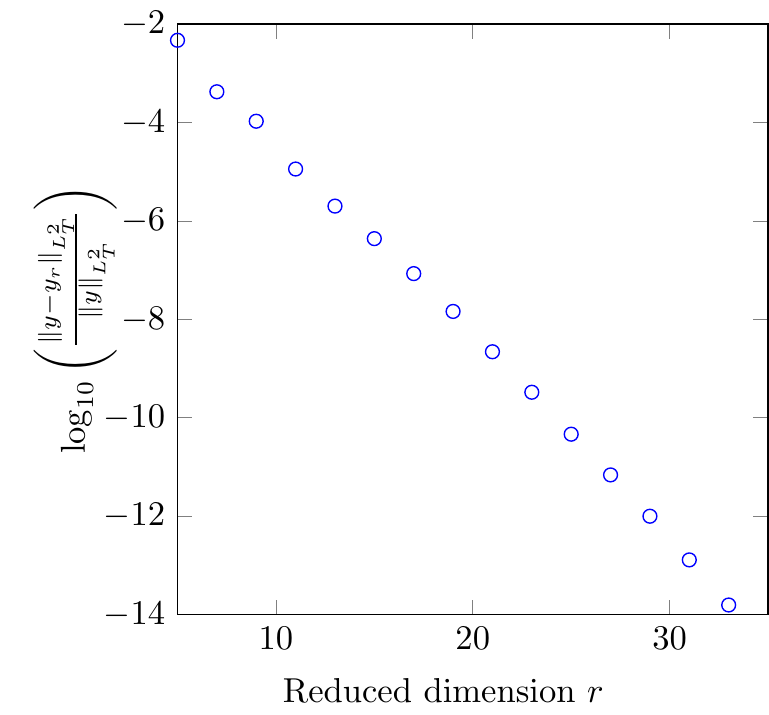}
 \caption{Logarithmic relative $L_T^2$-error for different reduced dimensions $r=5, 7, 9, \dots, 33$}\label{L2error}
 \end{minipage}\hspace{0.5cm}
 \begin{minipage}{0.45\linewidth}
  \hspace{-0.5cm}
 \includegraphics[width=1.1\textwidth,height=5.5cm]{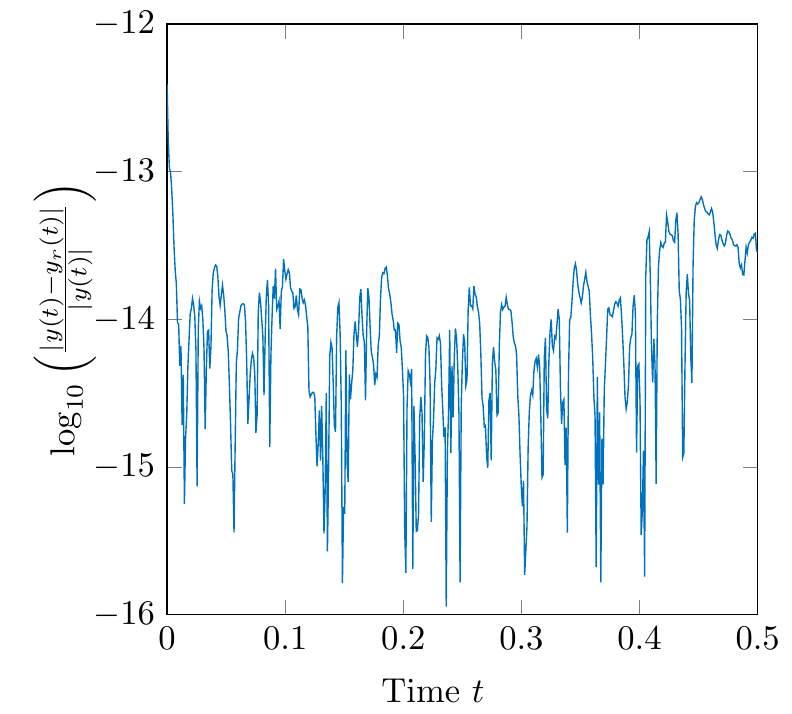}
 \caption{Logarithmic relative error in time for reduced dimension $r=33$.}\label{pointwise_error}
 \end{minipage}
 \end{figure}  
We conclude by explaining the time discretization used in order to obtain the simulation results. We implemented an implicit Runge-Kutta scheme for rough differential equations \cite{HHW18, rough_runge_kutta} with "optimal" rate solely based on the increments of the driver. Here, the implicit nature is required due to the stiffness of \eqref{finite_difference_model}. As a first step, we rewrite \eqref{stochstatenew} as \begin{align*}
             dx(t)=F\left(x(t)\right)d\tilde{\mathbf W}(t),\quad x(0)=x_0,
             \end{align*}
where $F(x):=\begin{bmatrix}Ax+f(x)  & N(x)K^{\frac{1}{2}}\end{bmatrix}$ and $\tilde W(t)=\begin{bmatrix} t  \\ W(t)\end{bmatrix}$. Given an equidistant partition $t_k=k \mathfrak h$ of $[0, T]$ with the step size $\mathfrak h$, we use the following scheme
\begin{align*}
 \begin{split}
      Z_{k, i} &= z_k + \sum_{j=1}^s a_{ij} F(Z_{k, j})  \big(\tilde W(t_{k+1})-\tilde W(t_k)\big) \\
      z_{k+1} &= z_k + \sum_{i=1}^s b_i F(Z_{k, i}) \big(\tilde W(t_{k+1})-\tilde W(t_k)\big) ,
 \end{split}
\end{align*}
with $z_0=x_0$ aiming that $z_k\approx x(t_k)$. In particular, Crouzeix's two stages ($s=2$) and diagonally implicit method is exploited that has the following Butcher tableau
\begin{align*}
\begin{array}{l|ll}
      & a_{11}&a_{12}\\  & a_{21} & a_{22}\\
      \hline
      & b_1 & b_2 \end{array}=\begin{array}{l|ll}
      & \frac{1}{2}+\frac{\sqrt{3}}{6}&0\\  & \frac{-\sqrt{3}}{3} & \frac{1}{2}+\frac{\sqrt{3}}{6}\\
      \hline
      & \frac{1}{2} & \frac{1}{2}  \end{array}.
\end{align*}
This method satisfies the optimality conditions provided in \cite{rough_runge_kutta} and hence has a convergence order arbitrary close to $2 H -0.5$, where $\frac{1}{3}<H\leq \frac{1}{2}$ is the Hurst index of a fractional Brownian motion. Now, let us mention that all the above simulations have been conducted setting $\mathfrak h=2^{-14}$.




\appendix
\section{Resolvent positive operators}\label{appendixbla}

This section covers the essential information on resolvent positive operators that are required in this paper. We refer to \cite{damm} for a more detailed and more general discussion. In particular, we are interested in such operators on
 $\left(H^n, \langle \cdot, \cdot\rangle_F\right)$ which shall be the Hilbert space of symmetric $n\times n$ matrices and $\langle \cdot, \cdot \rangle_F$ denotes the Frobenius inner product. Further suppose that  $H^n_+$ is the associated subset of symmetric positive semidefinite matrices. We begin with the definition of positive and resolvent positive operators on $H^n$.
\begin{definition}
A linear operator $\mathcal L: H^n\rightarrow H^n$ is called positive if $\mathcal L(H^n_+)\subset H^n_+$. It is resolvent positive if there is an $\alpha_0\in\mathbb R$ such that for all $\alpha>\alpha_0$ the operator $(\alpha I - \mathcal L)^{-1}$ is positive.
\end{definition}
The Lyapunov operator defined in \eqref{defn_L} is resolvent positive observing that it is a composition of a resolvent positive operator $X\mapsto A X + X A^\top$ and a positive part $X\mapsto \sum_{i, j=1}^d N_i X N_j^\top k_{ij}$, see \cite{damm}. Below, we state an equivalent characterization of resolvent positive operators and refer once more to \cite[Section 3.2.2]{damm} for a more general framework.
\begin{theorem}\label{equiresolpos}
A linear operator $\mathcal L: H^n\rightarrow H^n$ is resolvent positive if and only if $\langle V_1, V_2\rangle_F=0$ implies $\langle \mathcal L V_1, V_2\rangle_F\geq 0$ for $V_1, V_2\in H^n_+$.
\end{theorem}

\section*{Acknowledgments}
 MR is supported by the DFG via the individual grant ``Low-order approximations for large-scale problems arising in the context of high-dimensional
PDEs and spatially discretized SPDEs''-- project number 499366908.

\bibliographystyle{abbrv}
\bibliography{rough_mor}

\begin{thebibliography}{10}

\bibitem{antoulas}
A.~C. Antoulas.
\newblock {\em {Approximation of large-scale dynamical systems.}}
\newblock {Adv. in Design and Control. SIAM}, 2005.

\bibitem{BBRRS20}
C.~Bayer, D.~Belomestny, M.~Redmann, S.~Riedel, and J.~Schoenmakers.
\newblock Solving linear parabolic rough partial differential equations.
\newblock {\em J. Math. Anal. Appl.}, 490(1):124236, 45, 2020.

\bibitem{morbook2}
P.~Benner, A.~Cohen, M.~Ohlberger, and K.~Willcox, editors.
\newblock {\em Model reduction and approximation}, volume~15 of {\em
  Computational Science \& Engineering}.
\newblock Society for Industrial and Applied Mathematics (SIAM), Philadelphia,
  PA, 2017.
\newblock Theory and algorithms.

\bibitem{redmannbenner}
P.~{Benner} and M.~{Redmann}.
\newblock {Model Reduction for Stochastic Systems.}
\newblock {\em {Stoch PDE: Anal Comp}}, 3(3):291--338, 2015.

\bibitem{Bhatia1997}
R.~Bhatia.
\newblock {\em {Matrix Analysis}}, volume 169.
\newblock Springer, 1997.

\bibitem{CF09}
M.~Caruana and P.~Friz.
\newblock Partial differential equations driven by rough paths.
\newblock {\em J. Differential Equations}, 247(1):140--173, 2009.

\bibitem{CFO11}
M.~Caruana, P.~K. Friz, and H.~Oberhauser.
\newblock A (rough) pathwise approach to a class of non-linear stochastic
  partial differential equations.
\newblock {\em Ann. Inst. H. Poincar\'{e} C Anal. Non Lin\'{e}aire},
  28(1):27--46, 2011.

\bibitem{damm}
T.~Damm.
\newblock {\em {Rational Matrix Equations in Stochastic Control.}}
\newblock {Lecture Notes in Control and Information Sciences 297. Berlin:
  Springer}, 2004.

\bibitem{Dey12}
A.~Deya.
\newblock Numerical schemes for rough parabolic equations.
\newblock {\em Appl. Math. Optim.}, 65(2):253--292, 2012.

\bibitem{phd_diehl}
J.~Diehl.
\newblock {\em {Topics in stochastic differential equations and rough path
  theory}}.
\newblock PhD thesis, Technical University of Berlin, 2012.

\bibitem{DOR15}
J.~Diehl, H.~Oberhauser, and S.~Riedel.
\newblock A {L}\'{e}vy area between {B}rownian motion and rough paths with
  applications to robust nonlinear filtering and rough partial differential
  equations.
\newblock {\em Stochastic Process. Appl.}, 125(1):161--181, 2015.

\bibitem{FGLS17}
P.~K. Friz, P.~Gassiat, P.-L. Lions, and P.~E. Souganidis.
\newblock Eikonal equations and pathwise solutions to fully non-linear {SPDE}s.
\newblock {\em Stoch. Partial Differ. Equ. Anal. Comput.}, 5(2):256--277, 2017.

\bibitem{friz_hairer}
P.~K. Friz and M.~Hairer.
\newblock {\em A course on rough paths}.
\newblock Universitext. Springer, Cham, second edition, 2020.
\newblock With an introduction to regularity structures.

\bibitem{FV10}
P.~K. Friz and N.~B. Victoir.
\newblock {\em Multidimensional stochastic processes as rough paths}, volume
  120 of {\em Cambridge Studies in Advanced Mathematics}.
\newblock Cambridge University Press, Cambridge, 2010.
\newblock Theory and applications.

\bibitem{GLT06}
M.~Gubinelli, A.~Lejay, and S.~Tindel.
\newblock Young integrals and {SPDE}s.
\newblock {\em Potential Anal.}, 25(4):307--326, 2006.

\bibitem{GT10}
M.~Gubinelli and S.~Tindel.
\newblock Rough evolution equations.
\newblock {\em Ann. Probab.}, 38(1):1--75, 2010.

\bibitem{Hai11}
M.~Hairer.
\newblock Rough stochastic {PDE}s.
\newblock {\em Comm. Pure Appl. Math.}, 64(11):1547--1585, 2011.

\bibitem{Hai13}
M.~Hairer.
\newblock Solving the {KPZ} equation.
\newblock {\em Ann. of Math. (2)}, 178(2):559--664, 2013.

\bibitem{Hai14}
M.~Hairer.
\newblock A theory of regularity structures.
\newblock {\em Invent. Math.}, 198(2):269--504, 2014.

\bibitem{HM18}
M.~Hairer and K.~Matetski.
\newblock Discretisations of rough stochastic {PDE}s.
\newblock {\em Ann. Probab.}, 46(3):1651--1709, 2018.

\bibitem{HW13}
M.~Hairer and H.~Weber.
\newblock Rough {B}urgers-like equations with multiplicative noise.
\newblock {\em Probab. Theory Related Fields}, 155(1-2):71--126, 2013.

\bibitem{HHW18}
J.~Hong, C.~Huang, and X.~Wang.
\newblock Symplectic {R}unge-{K}utta methods for {H}amiltonian systems driven
  by {G}aussian rough paths.
\newblock {\em Appl. Numer. Math.}, 129:120--136, 2018.

\bibitem{staboriginal}
R.~Z. Khasminskii.
\newblock {Stochastic stability of differential equations.}
\newblock {Monographs and Textbooks on Mechanics of Solids and Fluids.
  Mechanics: Analysis, 7. Alphen aan den Rijn, The Netherlands; Rockville,
  Maryland, USA: Sijthoff \& Noordhoff.}, 1980.

\bibitem{Lyo98}
T.~J. Lyons.
\newblock Differential equations driven by rough signals.
\newblock {\em Rev. Mat. Iberoamericana}, 14(2):215--310, 1998.

\bibitem{LCL07}
T.~J. Lyons, M.~Caruana, and T.~L\'{e}vy.
\newblock {\em Differential equations driven by rough paths}, volume 1908 of
  {\em Lecture Notes in Mathematics}.

\bibitem{mao}
X.~Mao.
\newblock {\em Stochastic Differential Equations and Applications (Second
  Edition)}.
\newblock Woodhead Publishing, 2007.

\bibitem{redmannspa2}
M.~Redmann.
\newblock {Type II singular perturbation approximation for linear systems with
  Lévy noise}.
\newblock {\em SIAM J. Control Optim.}, 56(3):2120--2158., 2018.

\bibitem{h2_bilinear}
M.~Redmann.
\newblock {Bilinear Systems--A New Link to $\mathcal H_2$-norms, Relations to
  Stochastic Systems, and Further Properties}.
\newblock {\em SIAM J. Control Optim.}, 59(4), 2021.

\bibitem{rough_runge_kutta}
M.~Redmann and S.~Riedel.
\newblock {Runge-Kutta Methods for Rough Differential Equations}.
\newblock {\em {Journal of Stochastic Analysis}}, 3(4), 2022.

\bibitem{RS17}
S.~Riedel and M.~Scheutzow.
\newblock Rough differential equations with unbounded drift term.
\newblock {\em J. Differential Equations}, 262(1):283--312, 2017.

\end{thebibliography}
\end{document}